\documentclass[12pt,a4paper]{amsart}
\usepackage[utf8]{inputenc}
\usepackage[english]{babel}
\usepackage{amsmath,amsfonts,amssymb,amsthm}
\usepackage[dvipsnames]{xcolor}
\usepackage{mathtools}
\usepackage{tikz-cd}
\usepackage{subcaption}
\usepackage[percent]{overpic}

\usepackage[cal=euler]{mathalfa}
\usepackage{hyperref}
\hypersetup{
    colorlinks = true,
    linkcolor = {RedOrange},
    citecolor = {Green},
}

\usepackage[a4paper]{geometry}
\newtheorem{proposition}{Proposition}[section]
\newtheorem{theorem}[proposition]{Theorem}
\newtheorem{corollary}[proposition]{Corollary}
\newtheorem{lemma}[proposition]{Lemma}

\theoremstyle{remark}
\newtheorem{remark}[proposition]{Remark}

\newcommand{\lo}{\mathrm{o}}
\newcommand{\bO}{\mathrm{O}}

\newcommand{\ZZ}{\mathbb{Z}}

\newcommand{\CC}{\mathbb{C}}

\newcommand{\PSL}{\mathrm{PSL}}
\newcommand{\Seq}{\mathrm{Seq}}
\newcommand{\PSeq}{\mathrm{PSeq}}
\newcommand{\Nec}{\mathrm{Nec}}
\newcommand{\PNec}{\mathrm{PNec}}

\begin{document}

\title[Low-lying geodesics on the modular surface and necklaces]{Low-lying geodesics on the modular surface \\ and necklaces}

\author[Ara Basmajian]{Ara Basmajian}
\thanks{AB is supported by PSC CUNY Award 65245-00 53}
\address[Ara Basmajian]{The Graduate Center, CUNY, 365 Fifth Ave., N.Y., N.Y., 10016, USA,
and Hunter College, CUNY, 695 Park Ave., N.Y., N.Y., 10065, USA}
\email{abasmajian@gc.cuny.edu}

\author[Mingkun Liu]{Mingkun Liu}
\address[Mingkun Liu]{DMATH, FSTM, University of Luxembourg, Esch-sur-Alzette, Luxembourg}
\email{mingkun.liu@uni.lu}
\thanks{ML is supported by the Luxembourg National Research Fund OPEN grant O19/13865598}

\date{\today}

\begin{abstract}
    The $m$-thick part of the modular surface $X$ is the smallest compact subsurface of $X$ with horocycle boundary containing all the closed geodesics which wind around the cusp at most $m$ times.
    The $m$-thick parts form a compact exhaustion of $X$.
    We are interested in the geodesics that lie in the $m$-thick part (so called \emph{$m$ low-lying  geodesics}).
    We produce a complete asymptotic expansion for the number of $m$ low-lying geodesics of length equal to $2n$ in the modular surface.
    In particular, we obtain the asymptotic growth rate of the $m$ low-lying geodesics in terms of their word length using the natural generators of the modular group.
    After establishing a correspondence between this counting problem and the problem of counting necklaces with $n$ beads, we perform a careful singularity analysis on the associated generating function of the sequence. 
\end{abstract}

\maketitle

\section{Introduction}
\begin{figure}[ht]
    \centering
    \begin{overpic}[width=.7\textwidth]{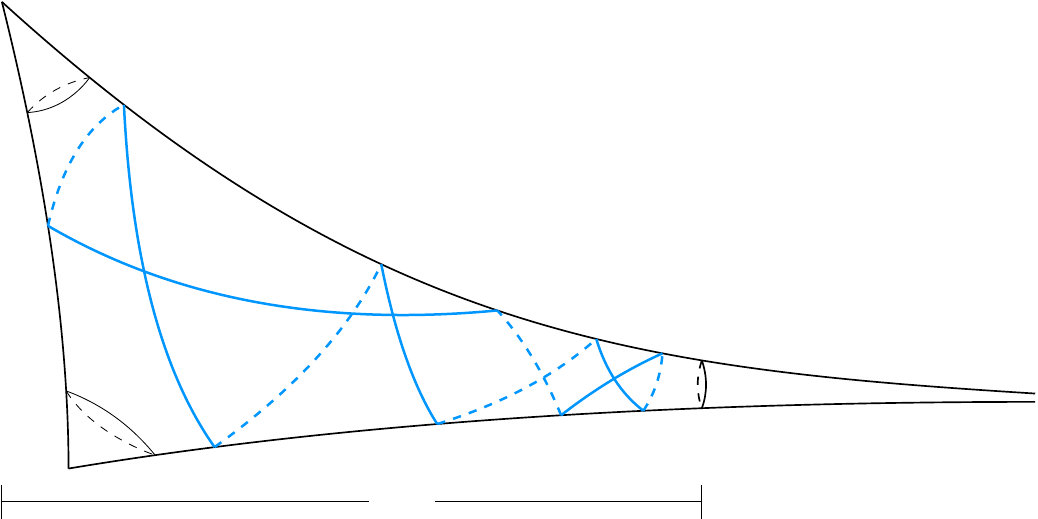}
        \put(40,30){$X$}
        \put(36.5,0.5){$X_m$}
    \end{overpic}
    \caption{An $m$ low-lying geodesic}
    \label{fig:lowlying}
\end{figure}
Consider the $(2,3,\infty)$ triangle group;
that is, the modular surface $X \coloneqq \mathbb{H}/\PSL(2, \ZZ)$.
There are many interesting classes of closed geodesics on $X$ including so-called reciprocal geodesics,
ones that stay in a fixed compact subsurface of $X$, and ones that exclusively leave a compact subsurface,
as well as of course the set of all closed geodesics on $X$.
Our interest in this paper is on the growth rate of the closed geodesics that stay in a fixed compact subsurface of $X$.
To be precise,
let $\CC \subset X$ be the cusp with its natural horocycle boundary  of length one in $X$.
For $m$ a positive integer, we define the \emph{$m$-thick part} of $X$, denoted $X_m$, to be the smallest compact subsurface of $X$ with horocycle boundary which contains all the closed geodesics which wind around the cusp at most $m$-times.
The $m$-thick parts form a compact exhaustion of $X$.
We are interested in the geodesics that lie in the $m$-thick part (so called \emph{$m$ low-lying  geodesics}).
See Figure~\ref{fig:lowlying}.

Using the fact that $\ZZ_2 \ast \ZZ_3$ is isomorphic to the modular group, we use word length with respect to the natural generators of the factors in $\ZZ_2 \ast \ZZ_3$ to study the growth of the $m$ low-lying geodesics.
In \cite{BasmajianValli1}, it was shown that 
\[
    |\{ \gamma \text{ a closed geodesic in $X_{m}$} : |\gamma| = 2n \}|
    \gtrsim
    \frac{2^{n(1-1/m)-1}}{n}
\]
as $n \to \infty$.
Our main result in this paper, Theorem~\ref{thm:main}, is a complete asymptotic analysis of the number of $m$ low-lying geodesics as well as the primitive ones.
Let $\alpha_m$ be the unique positive real solution of the equation $x^m - x^{m-1} - \cdots - x - 1 = 0$.
Let us mention that this equation is the characteristic equation of the generalized Fibonacci sequence, and $\alpha_m$ is a \emph{Pisot number}, namely, it is a real algebraic integer strictly greater than $1$, with all its Galois conjugates having modulus strictly less than $1$.

\begin{theorem} \label{thm:main}
    For any $k \in \ZZ_{\geq 1}$ and $m \in \ZZ_{\geq 3}$, when $n \to \infty$, we have
    \begin{equation} \label{eq:Aasymp}
        |\{ \gamma \text{ a closed geodesic in $X_{m}$} : |\gamma| = 2n \}|
        =
        \frac{1}{n}
        \sum_{d \mkern3mu \mid \mkern2mu n, \, d \leq k} \varphi(d) \, \alpha_m^{n/d} + 
        \bO \bigg( \frac{\alpha_m^{n/(k+1)}}{n} \bigg)
    \end{equation}
    where $\varphi$ is the Euler's totient function;
    and for primitive geodesics, we have
    \begin{equation} \label{eq:Basymp}
        |\{ \gamma \text{ a primitive closed geodesic in $X_{m}$} : |\gamma| = 2n \}|
        =
        \frac{1}{n}
        \sum_{d \mkern3mu \mid \mkern2mu n, \, d \leq k} \mu(d) \, \alpha_m^{n/d} + 
        \bO \bigg( \frac{\alpha_m^{n/(k+1)}}{n} \bigg)
    \end{equation}
    where $\mu$ stands for the Möbius function.
\end{theorem}

Some corollaries follow, 
\begin{corollary}
    For any $m \in \ZZ_{\geq 3}$,
    we have
    \[
        |\{ \gamma \text{ a closed geodesic in $X_{m}$} : |\gamma| = 2n \}|
        \sim
        \frac{\alpha_m^n}{n}
    \]
    as $n \to \infty$.
    The same conclusion holds for primitive closed geodesics.
\end{corollary}

\begin{corollary}
    The asymptotic growth rate of the primitive closed geodesics in $X_m$ converges to the asymptotic growth rate of the primitive closed geodesics on the modular orbifold, as $m \to \infty$.
\end{corollary}

There is an extensive literature on cusp excursions by a random geodesic on a hyperbolic surface including the papers \cite{Haas1, Haas2, Haas3, Haas4, Haas5, MelianPestana, Mor, Pollicott, RandeckerTiozzo, Stratmann, Sullivan, Trin}.
In particular, the papers \cite{Sullivan, Haas1, Haas2, Haas3, Haas4, Haas5} investigate the relation between depth, return time, and other invariants in various contexts including connections to number theory.
Papers involving growth of particular families of geodesics include \cite{BocaPasolPopaZaharescu, BasmajianValli2, Erlandsson, ErlandssonParlierSouto, ErlandssonSouto1, Mirzakhani, Sarnak}.
Geometric length growth of low-lying geodesics having the arithmetic condition of being fundamental is studied in \cite{BourgainKontorovich1, BourgainKontorovich2}. 
For hyperbolic geometry we use \cite{Buser} for a basic reference, and for combinatorial analysis \cite{FlajoletSedgewick}.

\subsection*{Acknowledgements}
We would like to thank Naomi~Bredon, Christian~El~Emam, Alex~Nolte, and Nathaniel~Sagman for useful discussions.
This project started during the first author's visit to the University of Luxembourg. It is a pleasure to thank the University of Luxembourg and Hugo~Parlier for their support and hospitality during that time.

\section{Necklaces and low-lying geodesics}
In this section we establish a correspondence between low-lying geodesics and  necklaces.  
A (binary) \emph{necklace} is made of a circular pattern of beads, each bead being one of two colors, say red or black, with the constraint that the number of consecutive adjacent beads of the same color being at most $m$.
Two necklaces are considered the same if they differ by a cyclic rotation.
It is not difficult to see that the set of all $m$ low-lying geodesics of length $n$ is in one-to-one correspondence with the set of necklaces made of $n$ beads with the longest run of the same color being at most $m$.
We denote the number of such necklaces of length $n$ by $A_m(n)$,
and by $B_m(n)$ the number of primitive ones.

In this paper, we give a complete asymptotic analysis of the number of $m$ low-lying geodesics.
Call the generator of the $\mathbb{Z}_2$ factor $a$, and the generator of the $\mathbb{Z}_3$ factor $b$.
Using the generators, $\{a,b,b^{-1}\}$ we define the length of a closed geodesic $\gamma$ on $X$, denoted $| \gamma |$, to be the minimal word length in the conjugacy class of a lift in $\PSL(2, \ZZ)$.
Any hyperbolic element can be conjugated into the normal form $ab^{\epsilon_1} ab^{\epsilon_2} \cdots ab^{\epsilon_n}$, and the normal forms realize the minimum word length;
hence the word length of a closed geodesic is necessarily even. 
Noting that the normal form is a product of parabolic elements (which represent going around the cusp), we are able to conclude how deep a geodesic wanders into the cusp by looking at the exponents of these parabolics.
Namely, staying in the compact subsurface $X_{m}$ is equivalent to not having a run of $\epsilon_{i}'s$ longer than $m$.
See for example Lemma~7.1 in \cite{BasmajianValli1} for a precise statement.
Hence there is a one-to-one correspondence between low-lying closed geodesics of length $2n$ and conjugacy classes of so called \emph{low-lying words} in the modular group.
Namely, a lift of a low-lying geodesic corresponds to a conjugacy class of hyperbolic elements in $\PSL(2, \ZZ)$.
Now such a conjugacy class has a normal form representative $ab^{\epsilon_1} ab^{\epsilon_2} \cdots ab^{\epsilon_n}$, where the number of consecutive $\epsilon_i$ of the same sign is at most $m$.
Assigning the color black to $+1$, and the color red to $-1$, we get a mapping 
\[
    (\epsilon_{1}, \dots, \epsilon_{n})
    \longmapsto
    ab^{\epsilon_1}ab^{\epsilon_2} \cdots ab^{\epsilon_n}
\]
which respects cyclic equivalence between the domain and range. 
We have shown,
\begin{proposition} \label{prop:geodToNecklace}
    For any $m \in \ZZ_{\geq 1}$, we have
    \[
        A_{m}(n)
        =
        |\{ \gamma \text{ a closed geodesic in $X_{m \geq 3}$} : |\gamma| = 2n \}|,
    \]
    and 
    \[
        B_{m}(n)
        =
        |\{ \gamma \text{ a primitive closed geodesic in $X_{m \geq 3}$} : |\gamma| = 2n \}|.
    \]
\end{proposition}

\section{Generating Functions}
Let $m \in \ZZ_{\geq 2}$.
For technical reasons, instead of working with $A_m(n)$ (resp.\,$B_m(n)$), we consider $\hat{A}_m(n)$ (resp.\,$\hat{B}_m(n)$), the number of $m$-necklaces (resp.\,primitive $m$-necklaces) of size $n$ that are \emph{nonconstant}.
We have
\[
    \hat{A}_m(n)
    =
    \begin{cases}
        A_m(n) - 2 & \text{if } 1 \leq n \leq m, \\
        A_m(n) & \text{if } n > m,
    \end{cases}
    \quad
    \text{and}
    \quad
    \hat{B}_m(n)
    =
    \begin{cases}
        B_m(n) - 2 = 0 & \text{if } n = 1, \\
        B_m(n) & \text{if } n > 1.
    \end{cases}
\]
In particular, $\hat{A}_m(n)$ (resp.\ $\hat{B}_m(n)$) and $A_m(n)$ (resp.\ $B_m(n)$) have the same asymptotic behavior.

We encode $\hat{A}_m(n)$ and $\hat{B}_m(n)$ into two generating functions $\Nec_m(z)$ and $\PNec_m(z)$, respectively, defined by
\[
    \Nec_m(z)
    =
    \sum_{n=1}^{\infty} \hat{A}_m(n) \, z^n,
    \qquad
    \PNec_m(z)
    =
    \sum_{n=1}^{\infty} \hat{B}_m(n) \, z^n.
\]
The numbers $\hat{A}_m(n)$ and $\hat{B}_m(n)$ can be recovered by extracting the coefficient of $z^n$ in the functions $\Nec_m(z)$ and $\PNec_m(z)$ respectively:
\[
    \hat{A}_m(n)
    =
    [z^n] \, \Nec_m(z),
    \qquad
    \hat{B}_m(n)
    =
    [z^n] \, \PNec_m(z)
\]
where $[z^n]$ stands for the coefficient extractor.

Define
\begin{equation} \label{eq:Fm}
    F_m(z)
    \coloneqq
    \frac{2 z^2 (1 - z^m) (m z^{m+1} -(m+1)z^m + 1)}{(z-1) (z^{m+1} - 1) (z^{m+1} - 2z + 1)}.
\end{equation}
\begin{proposition} \label{prop:Nec}
    We have formulas
    \begin{equation} \label{eq:PNec}
        \PNec_m(z)
        =
        \sum_{i=1}^{\infty} \mu(i) \int_0^1
        \frac{F_m((xz)^i)}{x} \, dx.
    \end{equation}
    and
    \begin{equation} \label{eq:Nec}
        \mathrm{Nec}_m(z)
        =
        \sum_{j=1}^{\infty}
        \sum_{i=1}^{\infty} \mu(i)
        \int_{0}^{1} 
        \frac{F_m((x z^j)^i)}{x} \, dx
    \end{equation}
    where $y = (xz^j)^i$ and $\mu$ stands for the Möbius function.
\end{proposition}
We proceed following \cite[Appendix~A.4]{FlajoletSedgewick}.
We say a binary sequence is a \emph{non-constant $m$-sequence} if it represents a non-constant  $m$-necklace, and we denote by $W_m(n)$ the number of non-constant $m$-sequences of size $n$.
For example, $W_2(4) = 6$, and the six non-constant $2$-sequences are: $(0,0,1,1)$, $(0,1,0,1)$, $(0,1,1,0)$, $(1,0,0,1)$, $(1,0,1,0)$, and $(1,1,0,0)$.
Note that $(0,0,1,0)$, $(0,1,0,0)$, $(1,0,1,1)$, and $(1,1,0,1)$ are \emph{not} non-constant $2$-sequences.
We denote by $\Seq_m(z)$ the generating function of $W_m(n)$, namely
\[
    \Seq_m(z)
    \coloneqq
    \sum_{n=1}^{\infty} W_m(n) \, z^n.
\]
\begin{lemma} \label{lem:Seq}
    We have formula
    \begin{equation}
        \Seq_m(z)
        =
        F_m(z)
    \end{equation}
\end{lemma}
\begin{proof}
    Every non-constant sequence can be decomposed into blocks of the same color such that adjacent blocks have different colors.
    For example, $(0,1,1,0,1,1,1,0,0)$ has $5$ blocks of sizes $1$, $2$, $1$, $3$, $2$, respectively.
    For non-constant $m$-sequences, they have a minimum of $2$ blocks, with each block size being bounded by $m$.
    If the first and the last block have the same color, the sum of their size is at most $m$.
    In binary sequences, the color of the first block determines the colors of the following ones, and the first and last block share the same color if and only if the number of blocks is odd.  

    Therefore, the generating function of non-constant $m$-sequences with even number of blocks is
    \[
        2\sum_{k=1}^{\infty} (z + z^2 + \cdots + z^m)^{2k}
        =
        2\sum_{k=1}^{\infty} \left( z \frac{1-z^m}{1-z} \right)^{2k}
        =
        -\frac{2z^2 (1 - z^m)^2}{(z^{m+1} - 1) (z^{m+1} - 2z + 1)},
    \]
    and the generating function of non-constant $m$-sequences with odd number of blocks is
    \begin{multline*}
        2\sum_{k=0}^{\infty}
        (z + z^2 + \cdots + z^m)^{2k+1} (z^2 + 2z^3 + \cdots + (m-1) z^m) \\
        =
        \frac{2z^2 ((m-1)z^m - mz^{m-1} + 1)}{(1-z)^2} \sum_{k=0}^{\infty} \left( z \frac{1-z^m}{1-z} \right)^{2k+1} \\
        =
        -\frac{2z^3 (z^m - 1) ((m-1)z^m - mz^{m-1} + 1)}{(z-1)(z^{m+1}-1) (z^{m+1} - 2z + 1)}.
    \end{multline*}
    Summing the two functions yields the lemma.
\end{proof}

\begin{proof}[Proof of Proposition~\ref{prop:Nec}]
    Let $\PSeq_m(z)$ denote the generating function of primitive $m$-sequences.
    By construction,
    \[
        \Seq_m(z)
        =
        \sum_{k=1}^{\infty} \PSeq_m(z^k).
    \]
    Note that this does not hold if constant sequences are included.
    Now it follows from the Möbius inversion formula and Lemma~\ref{lem:Seq} that
    \[
        \PSeq_m(z)
        =
        \sum_{k=1}^{\infty} \mu(k) \, \Seq_m(z^k)
        =
        \sum_{k=1}^{\infty} \mu(k) \, F_m(z^k)
    \]
    where $\mu$ is the Möbius function.

    Let $\PNec(z)$ be the generating function of primitive $m$-necklaces.
    We introduce an auxiliary variable $u$, and consider the bivariate generating functions $\PSeq(z,u) \coloneqq \PSeq(zu)$ and $\PNec(z,u) \coloneqq \PNec(zu)$.
    Observe that the primitive cycles of length $k$ and primitive sequences of length $k$ are in $1$-to-$k$ correspondence.
    Thus, in terms of generating functions, $\PNec(z, u)$ can be obtained by applying the transformation $u^k \mapsto u^k/k$ to $\PSeq(z, u)$, and equivalently,
    \[
        \PNec_m(z, u)
        =
        \int_0^u \PSeq_m(z, x) \, \frac{dx}{x}
        =
        \sum_{i=1}^{\infty} \mu(i) \int_0^u
        F_m((xz)^i) \, \frac{dx}{x}.
    \]
    Therefore, we obtain the formula
    \[
        \Nec_m(z, u)
        =
        \sum_{j=1}^{\infty} \PNec_m(z^j, u^j)
        =
        \sum_{j=1}^{\infty}
        \sum_{i=1}^{\infty} \mu(i) \int_0^{u^j}
        F_m((xz^j)^i) \, \frac{dx}{x}.
    \]
    Now the result follows by setting $u=1$.
\end{proof}

\section{Singularity analysis}
In this section, we perform the singularity analysis to track down the asymptotic behavior of $A_m(n)$ and $B_m(n)$.
Roughly speaking, rather than consider $\Nec_m(z)$ as a formal power series, we view it as a complex function.
Then, the asymptotic behavior of the coefficients $\hat{A}_m(n)$ of $\Nec_m(z)$ can be understood by analysing the behavior of $\Nec_m(z)$ near its singularities.
For details about this method, we recommend \cite[Chapter~VI]{FlajoletSedgewick}.

\begin{remark}
    Our approach is slightly different from the standard one introduced in \cite[Chapter~VI]{FlajoletSedgewick} as we take into account not only the principle singularities but also the minor singularities to obtain the complete asymptotic expansion \eqref{eq:Aasymp} and \eqref{eq:Basymp}.
\end{remark}

To prepare for it, let us begin with the following lemma which follows directly from results by Miles \cite{Miles}, Miller \cite{Miller}, and \cite{Wolfram}.

\begin{lemma}[{\cite{Miles, Miller, Wolfram}}] \label{lem:r}
    For any $m \in \ZZ_{\geq 2}$,
    apart from the trivial solution $z=1$, the polynomial equation
    \begin{equation} \label{eq:eq}
        1 - 2z + z^{m+1} = 0
    \end{equation}
    has exactly one positive real solution $r_m$ which is simple and lies in the interval $(1/2, 1)$.
    All other solutions have modulus strictly greater than $1$.
\end{lemma}
\begin{proof}
    Factoring $(z-1)$ from $1-2z+z^{m+1}$ yields
    \[
        1-2z+z^{m+1}
        =
        (z-1)(z^m + z^{m-1} + \cdots + z - 1).
    \]
    After the change of variables $z = 1/y$, the equation $z^m + z^{m-1} + \cdots + z - 1 = 0$
    becomes
    \begin{equation} \label{eq:Fibo}
        y^m - y^{m-1} - \cdots - 1 = 0.
    \end{equation}
    Now by \cite{Miles}, \cite{Miller}, or \cite[Lemma~3.6]{Wolfram}, equation \eqref{eq:Fibo} has a unique positive real solution in the interval $(1, 2)$, and all other solutions have modulus strictly greater than $1$.
    The lemma follows.
\end{proof}

\begin{figure}[ht]
    \centering
    \begin{subfigure}{.3\textwidth}
        \centering
        \includegraphics{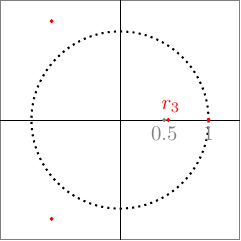}
        \caption{$m=3$}
    \end{subfigure}
    \begin{subfigure}{.3\textwidth}
        \centering
        \includegraphics{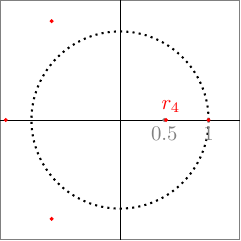}
        \caption{$m=4$}
    \end{subfigure}
    \begin{subfigure}{.3\textwidth}
        \centering
        \includegraphics{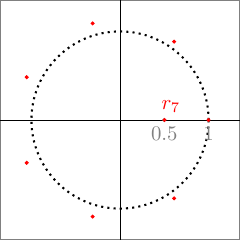}
        \caption{$m=7$}
    \end{subfigure}
    \caption{Roots of $1-2z+z^{m+1}$}
\end{figure}

\begin{remark}
    For $m=2,3$, we have exact formulas
    \[
        r_2
        =
        \frac{\sqrt{5}-1}{2},
        \qquad
        r_3
        =
        \frac{\sqrt[3]{3\sqrt{33}+17}}{3} - \frac{2}{3\sqrt[3]{3\sqrt{33} + 17}} - \frac{1}{3}.
    \]
    The exact expression of $r_4$ is already too lengthy to present here.
    Numerically,
    \[
        r_2 \approx 0.6180,
        \quad
        r_3 \approx 0.5437,
        \quad
        r_4 \approx 0.5188,
        \quad
        r_5 \approx 0.5087,
        \quad
        r_7 \approx 0.5020,
        \quad
        r_{10} \approx 0.5002.
    \]
    Asymptotically, when $m \to \infty$, we have $r_m \to 1/2$.
\end{remark}

\noindent\textbf{Notation.}
In order to simplify our notation, in the remainder of this section we 
write $r$ for $r_m$, $\alpha$ for $\alpha_m$, and  fix $m \in \ZZ_{\geq 2}$.

The idea is the following:
we write $\Nec_m(z)$ (and $\PNec_m(z)$) as the sum of two functions.
The first is a standard function that accounts for the main terms in the asymptotic expansion \eqref{eq:Aasymp}.
The other function corresponds to the error term in \eqref{eq:Aasymp}, and all its singularities are located farther from the origin than those of the first function.

Recall that we denote by $\mu$ the Möbius function, and by $\varphi$  Euler's totient function.
\begin{lemma} \label{lem:Nec2}
    For any $m \in \ZZ_{\geq 2}$, we have formulas
    \begin{equation}
        \Nec_m(z)
        =
        \sum_{k=1}^{\infty} \frac{\varphi(k)}{k} \log \frac{1}{1 - z^{k} / r}
        +
        h(z)
    \end{equation}
    and
    \begin{equation}
        \PNec_m(z)
        =
        \sum_{k=1}^{\infty} \frac{\mu(k)}{k} \log \frac{1}{1 - z^{k} / r}
        +
        h_{1}(z)
    \end{equation}
    where
    \begin{equation} \label{eq:h}
        h(z)
        \coloneqq
        \sum_{j=1}^{\infty} h_{j}(z),
        \qquad
        h_{j}(z)
        \coloneqq
        \sum_{i=1}^{\infty}
        \mu(i)
        \int_0^1 x^{i-1}z^{ij} \, f_m((x z^j)^i) \, dx,
    \end{equation}
    and
    \begin{equation} \label{eq:fm}
        f_m(z)
        \coloneqq
        \frac{2z(1-z^m)(mz^{m+1} - (m+1)z^m + 1) + (z-1)^2 (z^{m+1} - 1) \, \psi_m(z)}{(z - 1) (z^{m+1} - 1) (z^{m+1} - 2z + 1)}
    \end{equation}
    where $\psi_m$ is defined to be the unique polynomial such that
    \[
        t^{m+1} - 2t + 1
        =
        (t-1)(t^{m} + \cdots + t - 1)
        =
        (1-t) \cdot (r-t) \cdot \psi_m(t).
    \]
\end{lemma}
\begin{proof}
    By Proposition~\ref{prop:Nec}, we have
    \begin{equation} \label{eq:NecCopy}
        \Nec_m(z)
        =
        \sum_{j=1}^{\infty}
        \sum_{i=1}^{\infty} \mu(i)
        \int_{0}^{1} 
        \frac{F_m((x z^j)^i)}{x} \, dx
    \end{equation}
    where
    \[
        F_m(z)
        \coloneqq
        \frac{2 z^2 (1 - z^m) (m z^{m+1} -(m+1)z^m + 1)}{(z-1) (z^{m+1} - 1) (z^{m+1} - 2z + 1)}.
    \]
    Write $y \coloneqq (x z^j)^i$.
    Using the identity
    \[
        \log \frac{1}{1 - z^{ij}/r}
        =
        \int_0^1 \frac{i x^{i-1} z^{ij}}{r - x^i z^{ij}} \, dx,
        =
        \int_0^1 \frac{1}{x} \frac{y}{r-y},
    \]
    the integral that appears in the right-hand side of \eqref{eq:NecCopy} can be written as
    \[
        \int_{0}^{1} F_m(y) \, \frac{dx}{x}
        =
        \frac{1}{i} \log \frac{1}{1-z^{ij}/r}
        +
        \int_0^1 \left( \frac{F_m(y)}{x} - \frac{1}{x} \frac{y}{r - y} \right) dx.
    \]
    This can be further rewritten as
    \begin{multline*}
        \int_{0}^{1} F_m(y) \, \frac{dx}{x}
        =
        \frac{1}{i} \log \frac{1}{1-z^{ij}/r} \\
        +
        \int_0^1 x^{i-1}z^{ij} \, \frac{2y(1-y^m)(my^{m+1} - (m+1)y^m + 1) + (y-1)^2 (y^{m+1} - 1) \, \psi_m(y)}{(y - 1) (y^{m+1} - 1) (y^{m+1} - 2y + 1)} \, dx \\
        =
        \frac{1}{i} \log \frac{1}{1-z^{ij}/r}
        +
        \int_0^1 x^{i-1}z^{ij} \, f_m((x z^j)^i) \, dx.
    \end{multline*}
    Thus, the generating function $\Nec_m(z)$ equals
    \begin{equation} \label{eq:Nec2}
        \Nec_m(z)
        =
        \sum_{j=1}^{\infty}
        \sum_{i=1}^{\infty}
        \frac{\mu(i)}{i} \log \frac{1}{1 - z^{ij} / r}
        +
        \sum_{j=1}^{\infty} h_{j}(z) 
        =
        \sum_{k=1}^{\infty} \frac{\varphi(k)}{k} \log \frac{1}{1 - z^{k} / r}
        +
        h(z)
    \end{equation}
    where we have used, in the second equality, the identity (see \cite[Section~16.3]{HardyWright} for a proof)
    \[
        \sum_{d \mid k} \frac{\mu(d)}{d} = \frac{\varphi(k)}{k}
    \]
    and the fact that summing over the indices $i,j$ is the same as summing over $k$ and its divisors.

    Similarly, the generating function~\eqref{eq:PNec} can be rewritten as
    \begin{equation} \label{eq:PNec2}
        \PNec_m(z)
        =
        \sum_{k=1}^{\infty} \frac{\mu(k)}{k} \log \frac{1}{1 - z^{k} / r}
        +
        h_{1}(z)
    \end{equation}
    as claimed.
\end{proof}

The following lemma shows that the rational function $f_m(z)$ is holomorphic in $|z| < 1$.
\begin{lemma} \label{lem:sep}
    Let $r = r_m$ be as in Lemma~\ref{lem:r}.
    We have
    \[
        2r(1-r^m)(mr^{m+1} - (m+1) r^m + 1) + (r-1)^2 (r^{m+1} - 1) \, \psi_m(r)
        =
        0.
    \]
    In other words, the numerator of the rational function $f_m$ has $r$ as a root.
\end{lemma}
\begin{proof}
    Since $r^{m+1} - 2r + 1 = 0$,
    we have
    \[
        r(1-r^m)
        =
        1-r,
        \qquad
        r^{m+1} - 1
        =
        2(r-1).
    \]
    Thus, it suffices to show that
    \[
        m r^{m+1} - (m+1) \, r^m + 1 - (r-1)^2 \, \psi_m(r)
        =
        0.
    \]
    Note that for any polynomial $P$, if $x_0$ is a simple root of $P$ and $P(x) = (x-x_0) \, Q(x)$, then $Q(x_0) = P'(x_0)$.
    By Lemma~\ref{lem:r}, $r$ is a simple root of $1- 2x + x^{m+1}$, and hence
    \[
        (r-1) \, \psi_m(r)
        =
        (m+1) \, r^m - 2.
    \]
    On the other hand, again by the fact that $r^{m+1} - 2r + 1 = 0$, we have
    \begin{multline*}
        m r^{m+1} - (m+1) r^m + 1 \\
        =
        m r^{m+1} - (m+1) r^m + 1 + r^{m+1} - 2r + 1
        =
        ((m+1)r^m - 2) (r-1),
    \end{multline*}
    and the lemma follows.
\end{proof}

\begin{figure}[ht]
    \centering
    \includegraphics[width=.5\textwidth]{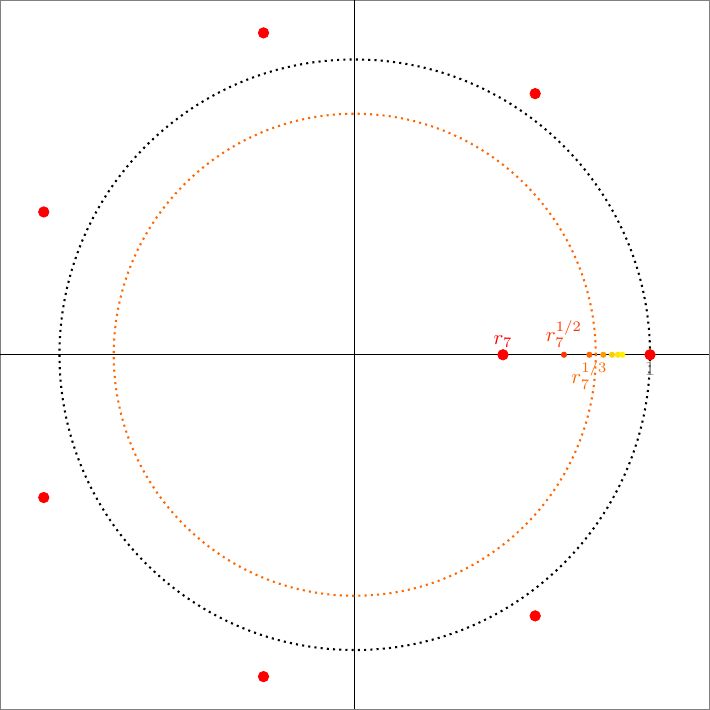}
    \caption{$g_3$ is holomorphic in the open disk of radius $r^{1/3}$}
\end{figure}

\begin{lemma} \label{lem:gk}
    For any $k \in \ZZ_{\geq 1}$, the functions defined by
    \begin{equation} \label{eq:g}
        g_k(z)
        \coloneqq
        \sum_{i=k}^{\infty} \frac{\varphi(i)}{i} \log\frac{1}{1-z^i/r},
        \qquad
        \tilde{g}_k(z)
        \coloneqq
        \sum_{i=k}^{\infty} \frac{\mu(i)}{i} \log\frac{1}{1-z^i/r},
        \qquad
    \end{equation}
    are holomorphic in the open disk $\{ z \in \CC : |z| < r^{1/k} \}$.
\end{lemma}
\begin{proof}
    Let $a \in (0,1)$.
    For any $z \in \CC$ such that $|z| \leq (ar)^{1/k}$, we have $|z|^i/r \leq a$ for any $i \geq k$, and hence
    \[
        |g_k(z)|
        \leq
        \sum_{i=k}^{\infty} \log\frac{1}{1-|z|^i/r}
        \leq
        \left( \frac{1}{ar} \log\frac{1}{1-a} \right) \sum_{i=k}^{\infty} |z|^i
        \leq
        \left( \frac{1}{ar} \log\frac{1}{1-a} \right) \sum_{i=k}^{\infty} (ar)^{i/k}
        <
        \infty
    \]
    where we have used the fact that (since $-\log(1-x)/x$ is increasing on $(0,1)$) for any $0 < x \leq x_0 < 1$,
    \[
        \log\frac{1}{1-x}
        \leq
        \frac{x}{x_0} \log\frac{1}{1-x_0}.
    \]
    Therefore, it follows from the Weierstrass M-test that the partial sum under consideration converges uniformly in the compact disk $\{ z \in \CC : |z| \leq (ar)^{1/k}\}$, where $a \in (0,1)$ is arbitrary.
    This implies the assertion for $g_k$.
    The same argument applies to $\tilde{g}_k$ as well.
\end{proof}

\begin{lemma} \label{lem:h}
    The functions $h_{1}(z)$ and $h(z)$ defined in \eqref{eq:h} are holomorphic in the open unit disk $\{ z \in \CC : |z| < 1 \}$.
\end{lemma}
\begin{proof}
    It follows from Lemma~\ref{lem:r} and \ref{lem:sep} that the rational function $f_m(y)$ defined by \eqref{eq:fm} is holomorphic in the open disk $|y|<1$.
    Thus, for any $a \in (0,1)$, there exists $M = M(a)$ such that $|f_m(z)| \leq M$ for any $z \in \CC$ with $|z| \leq a$, and therefore, for any $j \in \ZZ_{\geq 1}$ and $|z| \leq a$,
    \begin{equation} \label{eq:hj}
        |h_{j}(z)|
        \leq
        M \sum_{i=1}^{\infty} |z|^{ij} \int_0^1 x^{i-1} \, dx
        =
        M \sum_{i=1}^{\infty} \frac{|z|^{ij}}{i}
        =
        M \log \frac{1}{1-|z|^j}
        <
        \infty.
    \end{equation}
    In particular, it follows that $h_1$ is holomorphic in $|z| < 1$.
    Now by \eqref{eq:hj}, for any $|z| \leq a$, we have
    \[
        |h(z)|
        \leq
        \sum_{j=1}^{\infty} |h_{j}(z)|
        \leq
        M \sum_{j=1}^{\infty} \log \frac{1}{1-|z|^j}
        \leq
        \frac{M}{a} \log\frac{1}{1-a} \sum_{j=1}^{\infty} |z|^{j}
        <
        \infty
    \]
    which shows that $h(z)$ is holomorphic in the open disk $|z|<1$.
\end{proof}

Now, we are ready to prove our main result on necklace counting.
\begin{proof}[Proof of Theorem~\ref{thm:main}]
    First observe that Proposition~\ref{prop:geodToNecklace} allows us to translate the low-lying geodesic counting problem to counting necklaces.
    By Lemma~\ref{lem:Nec2}, for any $k \in \ZZ_{\geq 1}$, the generating function \eqref{eq:Nec} can be written as
    \begin{multline*}
        \Nec_m(z)
        =
        \varphi(1) \log\frac{1}{1-z/r} + \cdots + \frac{\varphi(k)}{k} \log\frac{1}{1-z^{k}/r} \\
        + \frac{\varphi(k+1)}{k+1} \log\frac{1}{1-z^{k+1}/r}
        + g_{k+2}(z)
        + h(z)
    \end{multline*}
    where $g_{k+2}$ is defined by \eqref{eq:g}.
    For any integer $1 \leq d \leq k+1$, we have
    \[
        [z^n] \, \log\frac{1}{1-z^d/r}
        =
        \begin{cases}
            \displaystyle \frac{(1/r)^{n/d}}{n/d} = d \, \frac{\alpha^{n/d}}{n} & \text{if } d \mid n \\
            0 & \text{otherwise}
        \end{cases}
    \]
    where $\alpha \coloneqq 1/r$.
    On the other hand, by Lemma~\ref{lem:gk} and \ref{lem:h}, the function $g_{k+2} + h$ is holomorphic in $\{ z \in \CC : |z| < r^{1/(k+2)} \}$ which contains the disk $\{ z \in \CC : |z| \leq r^{1/(k+1)} \}$.
    Thus it follows from Cauchy's inequality that
    \[
        [z^n] \, (g_{k+2}(z) + h(z))
        \leq
        \rho^{-n} \, \sup_{|z| = \rho} |g_{k+2}(z) + h(z)|
        =
        \bO(\rho^{-n})
        =
        \lo(\alpha^{n/(k+1)}/n)
    \]
    for any $r^{1/(k+1)} < \rho < r^{1/(k+2)}$.
    This completes the proof of \eqref{eq:Aasymp}.

    To prove \eqref{eq:Basymp}, it is sufficient to write the generating function \eqref{eq:PNec2} as
    \begin{multline*}
        \PNec_m(z)
        =
        \mu(1) \log\frac{1}{1-z/r} + \cdots + \frac{\mu(k)}{k} \log\frac{1}{1-z^{k}/r} \\
        + \frac{\mu(k+1)}{k+1} \log\frac{1}{1-z^{k+1}/r}
        + \tilde{g}_{k+2}(z)
        + h_{1}(z).
    \end{multline*}
    Then the rest of the proof is similar to that of \eqref{eq:Aasymp}.
\end{proof}

\section{Numerical computations}
Necklaces of small sizes, say $n \leq 25$, can be generated and counted using the SageMath package \href{https://doc.sagemath.org/html/en/reference/combinat/sage/combinat/necklace.html}{\texttt{sage.combinat.necklace}}.
For larger sizes, the CPU time required noticeably increases, as the necklace count grows exponentially.
However, we can still efficiently compute $A_m(n)$ and $B_m(n)$ using Proposition~\ref{prop:Nec}.
We have verified that the two methods agree for $n \leq 25$.

When $n$ is prime, $A_m(n) = B_m(n)$, and both can be well approximated by $\alpha_m^{n}/n$.
For instance, $A_3(83) = 111384745483589787826$ and $\alpha_3^{83}/83 \approx 111384745483589787826.0120$.
\begin{table}[ht]
    \begin{tabular}{l l l l l} 
        \hline \hline
         & $n=10$ & $n=20$ & $n=40$ & $n=80$ \\
         \hline
        $n^{-1} (\alpha_m^n$ & $\mathbf{4}4.3$ & $\mathbf{98}16.5$ & $\mathbf{9636}46499.3$ & $\mathbf{18572291511}299245526.4$ \\
        $+ \alpha_m^{n/2}$ & $\mathbf{4}6.4$ & $\mathbf{98}38.7$ & $\mathbf{9636514}07.5$ & $\mathbf{185722915117810}68776.1$ \\
        $+ 2\alpha_m^{n/4}$ & $\mathbf{47}.3$ & $\mathbf{984}0.8$ & $\mathbf{9636514}29.7$ & $\mathbf{1857229151178107}3684.3$ \\
        $+ 4\alpha_m^{n/5})$ & $\mathbf{4}8.7$ & $\mathbf{984}3.1$ & $\mathbf{96365144}2.8$ & $\mathbf{185722915117810745}42.1$ \\
        \hline
        $A_m(n)$ & $\mathbf{47}$ & $\mathbf{9844}$ & $\mathbf{963651447}$ & $\mathbf{18572291511781074575}$ \\
        \hline \hline
    \end{tabular}

    \bigskip

    \begin{tabular}{l l l l l} 
        \hline \hline
         & $n=10$ & $n=20$ & $n=40$ & $n=80$ \\
         \hline
        $n^{-1} (\alpha_m^n$ & $\mathbf{4}4.3$ & $\mathbf{9}816.5$ & $\mathbf{96364}6499.3$ & $\mathbf{1857229151}1299245526.4$ \\
        $- \alpha_m^{n/2}$ & $\mathbf{42}.2$ & $\mathbf{9794}.3$ & $\mathbf{9636415}91.0$ & $\mathbf{18572291510817422}276.8$ \\
        $- \alpha_m^{n/5}$ & $\mathbf{4}1.9$ & $\mathbf{979}3.8$ & $\mathbf{96364158}7.7$ & $\mathbf{1857229151081742206}2.4$ \\
        $+ \alpha_m^{n/10})$ & $\mathbf{42}.0$ & $\mathbf{9794}.0$ & $\mathbf{963641588}.0$ & $\mathbf{18572291510817422064}.0$ \\
        \hline
        $B_m(n)$ & $\mathbf{42}$ & $\mathbf{9794}$ & $\mathbf{963641588}$ & $\mathbf{18572291510817422064}$ \\
        \hline \hline
    \end{tabular}
    \caption{$m=3$ with exact digits in boldface}
\end{table}

\end{document}